\documentclass[11pt]{article}

\usepackage[T1]{fontenc}
\usepackage[a4paper,margin=1in]{geometry}
\usepackage{amsmath,amssymb,amsthm,mathtools,color}
\usepackage{microtype}
\usepackage[hidelinks]{hyperref}
\hypersetup{pdftitle={Bounded-VC chromatic thresholds of graphs},pdfauthor={Jinze Hu}}

\allowdisplaybreaks
\newtheorem{theorem}{Theorem}[section]
\newtheorem{lemma}[theorem]{Lemma}
\newtheorem{proposition}[theorem]{Proposition}

\theoremstyle{definition}

\theoremstyle{remark}

\newcommand{\VC}{\operatorname{VC}}
\newcommand{\N}{\mathcal{N}}
\newcommand{\M}{\mathcal{M}}
\newcommand{\K}{\mathcal{K}}
\newcommand{\A}{\mathcal{A}}
\newcommand{\B}{\mathcal{B}}
\newcommand{\F}{\mathcal{F}}
\newcommand{\eps}{\varepsilon}
\newcommand{\dvc}{\delta_{\chi}^{\mathrm{VC}}}
\newcommand{\vtx}{\operatorname{v}}

\title{The Bounded-VC chromatic thresholds of graphs\thanks{The work is supported by NSFC(11871015) and CGLSTDFFJ(2023L3003). E-mail: hujinze@fzu.edu.cn (J. Hu), qliu@fzu.edu.cn (Q. Liu), 1519036064@qq.com (L. Zhang), yhong@fzu.edu.cn (Y. Hong) .}}

\author{Jinze Hu$^a$, Qinghai Liu$^{a,c}$\footnote{Corresponding author.}, Liping Zhang$^{a}$, Yanmei Hong$^{b,c}$\\[1mm]
	{\small \it
		$^a$Center for Discrete Mathematics, Fuzhou University, Fuzhou, Fujian, 350108 China}\\
	{\small \it
		$^b$School of Mathematics and Statistics, Fuzhou University, Fuzhou, Fujian, 350108 China}\\
	{\small \it
		$^c$Key Laboratory for Operations Research and Cybernetics of Fujian Universities}
}

\date{}

\begin{document}
	\maketitle
	
	\begin{abstract}
		For a graph $H$, the chromatic threshold $\delta_\chi(H)$ is the infimum of $c>0$ such that the chromatic number of every $n$-vertex $H$-free graph with minimum degree at least $cn$ is bounded by a constant depending only on $H$ and $c$. Allen, B\"ottcher, Griffiths, Kohayakawa, and Morris proved that if $\chi(H)=r\geq 3$, then
		$\delta_{\chi}(H)\in\{\frac{r-3}{r-2}, \frac{2r-5}{2r-3}, \frac{r-2}{r-1}\}$.
		Liu, Shangguan, Skokan, and Xu introduced the bounded-VC chromatic threshold $\dvc(H)$ by restricting the host graphs to have bounded VC-dimension. We determine this parameter for graph $H$ with $\chi(H)\ge 3$. More precisely, let $\M(H)$ be the decomposition family of an $r$-chromatic graph $H$, then
		\[
		\dvc(H)=
		\begin{cases}
			\dfrac{r-3}{r-2},&\text{if $\M(H)$ contains a forest},\\[4pt]
			\dfrac{r-2}{r-1},&\text{otherwise}.
		\end{cases}
		\]
	\end{abstract}
	
	
	
	\section{Introduction}

	The {\it chromatic number} of a graph $G$, a well-studied concept in graph theory, means the minimum number of colors as an assignment to the vertices of $G$ such that no two adjacent vertices are assigned the same color. Usually, we  denote the chromatic number of $G$ by  $\chi(G)$ and we call $G$ a $\chi(G)$-chromatic graph. A classical problem related to the chromatic number is to explore the structure of graphs with a forbidden subgraph. The classical constructions of Zykov~\cite{Zykov} and Tutte~\cite{Tutte} show that forbidding a triangle does not bound the chromatic number. This result was later strengthened by Erd\H{o}s~\cite{ErdosProbability}, who constructed graphs with arbitrarily large girth and chromatic number. This implies that there exists a graph with a forbidden non-acyclic graph admitting unbounded chromatic number.
	
	Erd\H{o}s and Simonovits~\cite{ErdosSimonovits} asked whether a linear minimum-degree condition bounds the chromatic number. This question led to the chromatic threshold of a graph $H$, defined by
	\[
	\begin{aligned}
		\delta_\chi(H):=\inf\bigl\{c\geq 0:&\text{ there exists }C=C(H,c)\text{ such that every $n$-vertex $H$-free graph $G$}\\& \text{ with }  \delta(G)\geq cn\text{ satisfies }\chi(G)\leq C\bigr\}.
	\end{aligned}
	\]
	For triangles, Erd\H{o}s and Simonovits conjectured that $\delta_\chi(K_3)=1/3$, which was proved by Thomassen~\cite{ThomassenTriangle}. Goddard and Lyle~\cite{GoddardLyle}, and independently Nikiforov~\cite{Nikiforov}, determined the threshold for every complete graph. The general problem was resolved by Allen, B\"ottcher, Griffiths, Kohayakawa, and Morris~\cite{AllenEtAl}: if $\chi(H)=r\geq 3$, then
	\begin{equation}\label{eq:ordinary-values}
		\delta_\chi(H)\in
		\left\{
		\frac{r-3}{r-2},
		\frac{2r-5}{2r-3},
		\frac{r-2}{r-1}
		\right\}.
	\end{equation}
	Moreover, they characterized exactly which graphs attain each of the three values.
	
	The \textit{Vapnik--Chervonenkis dimension} is a measure of the complexity of a set system~\cite{Sauer,Shelah,VapnikChervonenkis}. For a graph $G$, it is natural to apply this notion to the family of vertex neighborhoods. \L{}uczak and Thomass\'e~\cite{LuczakThomasse} initiated the use of VC-dimension methods in chromatic-threshold problems. More recently, Liu, Shangguan, Skokan, and Xu~\cite{LiuEtAl} introduced the bounded-VC chromatic threshold
	\[
	\begin{aligned}
		\dvc(H):=\inf\bigl\{c\geq 0:&\ \forall d\in\mathbb N,\ \exists C=C(H,c,d)\text{ such that}
		\text{ every $n$-vertex $H$-free graph $G$}\\& \text{with }\VC(G)\leq d\text{ and }\delta(G)\geq cn
		\text{ satisfies }\chi(G)\leq C\bigr\}.
	\end{aligned}
	\]
	Clearly, $\dvc(H)\leq\delta_\chi(H)$. Kim, Liu, Shangguan, Wang, Wu, and Xue~\cite{KimEtAl} subsequently determined the bounded-VC threshold of a clique, proving that for $r\ge 3$, $\dvc(K_r)=\frac{r-3}{r-2}$.
	The result suggests that the bounded VC-dimension can weaken the minimum-degree condition required for bounding the chromatic number. The natural problem is therefore to determine which structural feature of $H$ governs this improvement.
	
	For an $r$-chromatic graph $H$, its \emph{decomposition family} $\M(H)$ is the family of bipartite graphs obtained from $H$ by deleting $r-2$ color classes in some proper $r$-coloring of $H$. 
	Our main result shows that, for the bounded-VC chromatic thresholds, this decomposition family is the only structural ingredient that remains.
	
	\begin{theorem}\label{thm:main}
		Let $H$ be a graph with $\chi(H)=r\geq 3$. Then
		\begin{equation}\label{eq:classification}
			\dvc(H)=
			\begin{cases}
				\dfrac{r-3}{r-2},&\text{if $\M(H)$ contains a forest},\\[6pt]
				\dfrac{r-2}{r-1},&\text{otherwise}.
			\end{cases}
		\end{equation}
	\end{theorem}
	
	Theorem~\ref{thm:main} gives a complete two-valued classification. In comparison with~\eqref{eq:ordinary-values}, the intermediate value $(2r-5)/(2r-3)$ disappears. More precisely, Allen, B\"ottcher, Griffiths, Kohayakawa, and Morris~\cite{AllenEtAl} proved that $\delta_\chi(H)=(r-2)/(r-1)$ precisely when $\M(H)$ contains no forest, and that $\delta_\chi(H)=(r-3)/(r-2)$ precisely when $H$ is $r$-near-acyclic. Hence bounded VC-dimension changes the threshold exactly for those graphs whose decomposition family contains a forest but which are not $r$-near-acyclic.
	Several immediate consequences are worth recording. Since $K_2\in\M(K_r)$, Theorem~\ref{thm:main} recovers the clique result of Kim et al.~\cite{KimEtAl}. If $\M(H)$ contains no forest, then the bounded-VC threshold coincides with both the ordinary chromatic threshold and the Tur\'an density of $H$.
	
	A convenient reformulation of the forest condition explains the form of our upper-bound argument. For graphs $H_1$ and $H_2$, we write $H_1\vee H_2$ for their \textit{join}. For integers $r,t\geq 1$, let $K_r(t)$ be the complete $r$-partite graph with every part of order $t$. Proposition~\ref{prop:structural} shows that $\M(H)$ contains a forest if and only if $H\subseteq F\vee K_{r-2}(t)$ for some forest $F$. Thus, in the first case of Theorem~\ref{thm:main}, it suffices to find a bounded cover of $V(G)$ by common neighborhoods of copies of $K_{r-2}(t)$. Then each member of this cover must induce an $F$-free graph and therefore has bounded chromatic number.
	
	
	The paper is organized as follows. Section~\ref{sec:preliminaries} contains the VC-dimension tools and the required graph-theoretic facts. In Section~\ref{sec:upper} we prove the upper bounds, treating the case $r=3$ separately before establishing the common-neighborhood covering lemma for $r\geq 4$. Section~\ref{sec:lower} gives the bounded-VC lower-bound constructions and completes the proof of Theorem~\ref{thm:main}.
	
	We follow standard notation through. 
	Let $G$ be a graph. 
	For $S\subseteq V(G)$, $G[S]$ is the subgraph of $G$ induced by $S$. 
	Let $N_S(v)$ denote the neighborhood of vertex $v$ in $G[S]$ and
	$\deg_S(v)$ denote the cardinality of $N_S(v)$. 
	Usually, we use $N_G(v)$ and $\deg_G(v)$ to denote the neighborhood of $v$ in $G$ and the cardinality of $N_G(v)$, respectively.

	\section{Preliminaries}\label{sec:preliminaries}
	
	We first record the structural observation used throughout our proof.
	
	\begin{proposition}\label{prop:structural}
		Let $H$ be a graph with $\chi(H)=r\geq 3$. The following are equivalent.
		\begin{enumerate}
			\item[(i)] The decomposition family $\M(H)$ contains a forest.
			\item[(ii)] There exist a forest $F$ and an integer $t\geq 1$ such that $H\subseteq F\vee K_{r-2}(t)$.
		\end{enumerate}
	\end{proposition}
	
	\begin{proof}
		Suppose first that $\M(H)$ contains a forest. Choose a proper $r$-coloring of $H$ with color classes $V_1,\ldots,V_r$ such that $H[V_1\cup V_2]$ is a fores. Let
		\[
		F:=H[V_1\cup V_2]
		\qquad
		\text{and}
		\qquad
		t:=\max\{|V_i|:3\leq i\leq r\}.
		\]
		Embedding each $V_i$, $3\leq i\leq r$, into a distinct part of $K_{r-2}(t)$ shows that $H\subseteq F\vee K_{r-2}(t)$.
		
		Conversely, suppose that $H\subseteq F\vee K_{r-2}(t)$ for a forest $F$. Color the two sides of a bipartition of $F$ with two colors and use one additional color for each part of $K_{r-2}(t)$. Restricting this coloring to $H$ gives a proper $r$-coloring. After deleting the $r-2$ color classes lying in $K_{r-2}(t)$, the remaining graph is a subgraph of $F$, and hence is a forest. Therefore $\M(H)$ contains a forest.
	\end{proof}
	
	\subsection{VC-dimension and transversals}
	
	Let $\F\subseteq 2^X$ be a set system. For $Y\subseteq X$, the \textit{trace} of $\F$ on $Y$ is $\F|_Y:=\{F\cap Y:F\in\F\}$.
	We say that $Y$ is \emph{shattered} by $\F$ if $\F|_Y=2^Y$. The \textit{VC-dimension} of $\F$, denoted by  $\VC(\F)$, is the maximum order of a set shattered by $\F$. For a graph $G$, put
	\[
	\N(G):=\{N_G(v):v\in V(G)\}
	\qquad
	\text{and}
	\qquad
	\VC(G):=\VC(\N(G)).
	\]
	A \emph{transversal} of $\F$ is a subset of $X$ meeting every member of $\F$, and $\tau(\F)$ denotes the minimum order of a transversal. A \emph{fractional transversal} is a function $\omega:X\to[0,1]$ such that for every $F\in\F$,
	$\sum_{x\in F}\omega(x)\geq 1$.
	The minimum possible value of $\sum_{x\in X}\omega(x)$ is denoted by $\tau^*(\F)$. We use the following standard consequence of the $\eps$-net theorem of Haussler and Welzl~\cite{HausslerWelzl}.
	
	\begin{lemma}(\cite{HausslerWelzl})\label{lem:epsilon-net}
		Every set system $\F$ with $\VC(\F)\leq d$ satisfies
		$\tau(\F)\leq 16d\tau^*(\F)\log\bigl(2d\tau^*(\F)\bigr)$.
	\end{lemma}
	
	For an indexed set system $\F=\{F_i:i\in I\}\subseteq 2^X$, its \textit{dual} is the set system $\F^*$ on ground set $I$ whose members are
	\[
	F_x^*:=\{i\in I:x\in F_i\},\]
	for $x\in X$.
	By a classical result of Assouad \cite{Assouad}, bounded VC-dimension is preserved under taking the dual set system.
	\begin{lemma}\label{lem:dual}(\cite{Assouad})
		If $\VC(\F)\leq D$, then $\VC(\F^*)\le 2^{D+1}-1$.
	\end{lemma}
	
	We shall also use the Sauer--Shelah lemma~\cite{Sauer,Shelah}.
	
	\begin{lemma}[Sauer--Shelah~\cite{Sauer,Shelah}]\label{lem:sauer}
		Every set system $\F$ with $\VC(\F)\leq d$, then for every $m$-element set $Y\subseteq X$,
		\[
		|\F|_Y|\leq\sum_{i=0}^{d}\binom{m}{i}.
		\]
		In particular, if $m\geq d\geq 1$, then
		$|\F|_Y|\leq \left(\frac{em}{d}\right)^d$.
	\end{lemma}
	
	The following lemma establishes bounds on set systems generated by common neighborhoods.
	
	\begin{lemma}\label{lem:intersections}
		Let $\N\subseteq 2^X$ with $\VC(\N)\leq d$ and for $q\geq 1$, let
		\[
		\A:=\{N_1\cap\cdots\cap N_q:N_1,\ldots,N_q\in\N\}.
		\]
		Then $\VC(\A)\leq D(d,q)$ for some constant $D(d,q)$ depending only on $d$ and $q$.
	\end{lemma}
	
	\begin{proof}
		Let $Y$ be an $m$-element set shattered by $\A$. If $m<d$, there is nothing to prove. Otherwise, every trace in $\A|_Y$ is the intersection of $q$ traces in $\N|_Y$. 
		For every $A\in\mathcal A$, there exist
		$N_1,\ldots,N_q\in\mathcal N$ such that
		$A=N_1\cap\cdots\cap N_q$.
		Hence, 
		\[
		A\cap Y
		=(N_1\cap\cdots\cap N_q)\cap Y
		=(N_1\cap Y)\cap\cdots\cap(N_q\cap Y).
		\]
		It follows that every member of $\mathcal A|_Y$is the intersection of
		$q$ members of $\mathcal N|_Y$. Then
		there are at most
		$\lvert\mathcal N|_Y\rvert^q$ ordered \(q\)-tuples of such traces.
		Different $q$-tuples may yield the same intersection, and consequently
		$\lvert\mathcal A|_Y\rvert
		\leq
		\lvert\mathcal N|_Y\rvert^q$.
		By Lemma~\ref{lem:sauer}, we have
		\[
		2^m=|\A|_Y|
		\leq |\N|_Y|^q
		\leq \left(\frac{em}{d}\right)^{dq}.
		\]
		The left-hand side grows exponentially in $m$, whereas the right-hand side is a polynomial of degree $dq$. Hence, $m$ is bounded in terms of $d$ and $q$ only.
	\end{proof}


	\subsection{Graph-theoretic tools}
	
	We use two standard facts. The first is an elementary degeneracy argument. The second follows from the Erd\H{o}s--Simonovits supersaturation theorem~\cite{ErdosSimonovitsSupersaturation}.
	For convenience, we write $\vtx(F):=|V(F)|$ for the order of a graph $F$.
	
	\begin{lemma}\label{lem:forest-color}
		If $F$ is a forest and $G$ is $F$-free, then
		$\chi(G)\leq \vtx(F)$.
	\end{lemma}
	
	\begin{proof}
		Every subgraph of $G$ contains a vertex of degree smaller than $\vtx(F)$; otherwise a greedy embedding, one component at a time, would produce a copy of $F$. Hence $G$ is $(\vtx(F)-1)$-degenerate and is therefore $\vtx(F)$-colorable.
	\end{proof}
	
	\begin{lemma}(\cite{ErdosSimonovitsSupersaturation})\label{lem:supersaturation}
		Let $r\geq 2$, $t\geq 1$, and $\eta>0$. There exist constants $\gamma=\gamma(r,t,\eta)>0$ and $n_0=n_0(r,t,\eta)$ such that every graph $G$ on $n\geq n_0$ vertices with
		$\delta(G)\geq\left(\frac{r-2}{r-1}+\eta\right)n$
		contains at least $\gamma n^{rt}$ labeled copies of $K_r(t)$.
	\end{lemma}
	
	\section{The upper bound}\label{sec:upper}
	
	We first prove the upper bound in the case where $\M(H)$ contains a forest. By Proposition~\ref{prop:structural}, we may choose a forest $F$ and an integer $t\geq 1$ such that
	\begin{equation}\label{eq:ambient-join}
		H\subseteq F\vee K_{r-2}(t).
	\end{equation}
	The case $r=3$ requires only the $\eps$-net theorem. For $r\geq 4$, we combine it with supersaturation.
	
	\subsection{The case \texorpdfstring{$r=3$}{r=3}}
	We denote an independent set with $t$ vertices by $I_t$.
	\begin{theorem}\label{thm:r3}
		Let $t\geq 1$, $F$ be a forest, and suppose that $H\subseteq F\vee I_t$. For every $c>0$ and $d\in\mathbb N$, there is a constant $C=C(F,t,c,d)$ such that every $n$-vertex $H$-free graph $G$ with $\VC(G)\leq d$ and $\delta(G)\geq cn$ satisfies $\chi(G)\leq C$.
	\end{theorem}
	
	\begin{proof}
		We may assume $0<c\leq 1$. Choose $n_0=n_0(t,c)$ sufficiently large. If $n<n_0$, then $\chi(G)\leq n_0$, so assume that $n\geq n_0$.
		
		Assign every vertex of $G$ independently and uniformly to one of $t+1$ parts $V_0,\ldots,V_t$. For every vertex $v$ and every $i$, we have  $\mathbb{E}[\deg_{V_i}(v)]=\deg_G(v)/(t+1)$. Chernoff's inequality gives
		\[
		\mathbb{P}\left[\deg_{V_i}(v)<\frac{\deg_G(v)}{2(t+1)}\right]
		\leq \exp\left(-\frac{\deg_G(v)}{8(t+1)}\right)
		\leq \exp\left(-\frac{cn}{8(t+1)}\right).
		\]
		For $n_0$ sufficiently large, we deduce that
		\[
		\sum_{v\in V(G)}\mathbb{P}[\deg_{V_i}(v)<\frac{cn}{2(t+1)}]
		\le n\cdot \exp\left(-\frac{cn}{8(t+1)}\right)=o(1).
		\]
		Thus, there exists such a partition of $V(G)$ for which
		\begin{equation}\label{eq:balanced-neighborhoods}
			\deg_{V_i}(v)\geq\frac{cn}{2(t+1)}
		\end{equation}
		for every $v\in V(G)$ and $0\leq i\leq t$.
		
		For distinct $i,j\in\{0,\ldots,t\}$, define
		\[
		\F_i^j:=\{N_G(v)\cap V_j:v\in V_i\}.
		\]
		Since $\F_i^j$ is obtained by restricting members of $\N(G)$ to $V_j$, we have $\VC(\F_i^j)\leq d$. Give every vertex of $V_j$ weight $2(t+1)/(cn)$. By~\eqref{eq:balanced-neighborhoods}, every member of $\F_i^j$ has weight at least one, while the total weight is at most $k:=2(t+1)/c$.
		Thus $\tau^*(\F_i^j)\leq k$. Lemma~\ref{lem:epsilon-net} yields a transversal $T_i^j\subseteq V_j$ of order at most $L:=16dk\log(2dk)$.
		
		For a fixed $i$, let $\mathcal T_i$ be the family of $t$-element sets obtained by choosing one vertex from each $T_i^j$ with $j\neq i$. Since the parts $V_j$ are disjoint, these choices give $t$ distinct vertices, and $|\mathcal T_i|\leq L^t$.
		For every $v\in V_i$ and every $j\neq i$, the transversal $T_i^j$ meets $N_G(v)\cap V_j$. Hence $v$ belongs to the common neighborhood, inside $V_i$, of some $T\in\mathcal T_i$. Therefore
		\[
		V_i=\bigcup_{T\in\mathcal T_i}\left(V_i\cap\bigcap_{x\in T}N_G(x)\right).
		\]
		Each graph induced by a set in this union is $F$-free. Indeed, a copy of $F$ together with the vertices of $T$ would contain $F\vee I_t$ as a subgraph, and hence would contain $H$, contrary to the assumption that $G$ is $H$-free. By Lemma~\ref{lem:forest-color}, each member of the cover induces a graph of chromatic number at most $\vtx(F)$. Consequently,
		\[
		\chi(G[V_i])\leq L^t\vtx(F)
		\]
		for every $i$,
		and thus
		\[
		\chi(G)\leq(t+1)L^t\vtx(F).
		\]
	\end{proof}
	
	\subsection{A common-neighborhood covering lemma}
	
	For a copy $K$ of a given graph in $G$, write
	\[
	A_K:=\bigcap_{v\in V(K)}N_G(v)
	\]
	for its common neighborhood.
	
	\begin{lemma}\label{lem:cover}
		Let $r\geq 4$, $t\geq 1$, $c>0$, and $d\in\mathbb N$. There exist constants $C_0=C_0(r,t,c,d)$ and $n_0=n_0(r,t,c)$ such that every $n$-vertex graph $G$ with $n\geq n_0$, $\VC(G)\leq d$, and
		$\delta(G)\geq\left(\frac{r-3}{r-2}+c\right)n$
		contains copies $K_1,\ldots,K_m$ of $K_{r-2}(t)$, where $m\leq C_0$, such that $V(G)=\bigcup_{i=1}^{m}A_{K_i}$.
	\end{lemma}
	
	\begin{proof}
		Put $s:=r-2$ and $\alpha:=\frac{s-1}{s}+c$.
		If $\alpha\geq 1$, then no such graph exists, so assume $\alpha<1$. Fix $x\in V(G)$, let $M_x:=G[N_G(x)]$, and set $m_x:=|N_G(x)|$. For every $y\in N_G(x)$,
		\[
		\deg_{M_x}(y)
		\geq \deg_G(y)+\deg_G(x)-n\geq \alpha n+m_x-n
		\geq \left(2-\frac1\alpha\right)m_x,
		\]
		where the last inequality follows from $m_x\geq\alpha n$. Furthermore,
		\[
		2-\frac1\alpha
		=
		\frac{s-2}{s-1}+\eta,\]
		where $\eta:=\frac{cs^2}{(s-1)(s-1+cs)}>0$.
		Since $m_x\geq\alpha n$ and $n_0$ is sufficiently large, Lemma~\ref{lem:supersaturation} implies that $M_x$ contains at least $\gamma m_x^{st}\geq\gamma\alpha^{st}n^{st}$
		labeled copies of $K_s(t)$, where $\gamma>0$ depends only on $r,t,c$. 
		
		Let $q:=st=(r-2)t$, and let $\K$ be the indexed family of all labeled copies of $K_s(t)$ in $G$. For each $x\in V(G)$, put
		\[
		B_x:=\{K\in\K:K\subseteq N_G(x)\}.
		\]
		It follows from $|\K|\leq n^q$, that
		$|B_x|\geq\gamma\alpha^qn^q\geq\eps|\K|$,
		where  $\eps:=\gamma\alpha^q>0$.
		Let 
		\[\A:=\{A_K:K\in\K\}\] be an indexed family  on ground set $V(G)$ and the family $\B:=\{B_x:x\in V(G)\}$ on ground set $\K$. Then, we have $x\in A_K$ if and only if $K\in B_x$, which shows that $\B$ is the dual system of $\A$. Every $A_K$ is an intersection of $q$ vertex neighborhoods. By means of Lemma~\ref{lem:intersections}, we arrive at $\VC(\A)\leq D(d,q)$.
		Together with Lemma~\ref{lem:dual}, it yields $\VC(\B)<2^{D(d,q)+1}=:D'$.
		
		Give each $K\in\K$ weight $1/(\eps|\K|)$. Every $B_x$ has weight at least one, and the total weight is $1/\eps$. Thus $\tau^*(\B)\leq1/\eps$. Lemma~\ref{lem:epsilon-net} gives
		\[
		\tau(\B)
		\leq 16D'\eps^{-1}\log(2D'\eps^{-1})
		=:C_0.
		\]
		Then there exists a transversal $\K_0\subseteq\K$ of $\B$ with $|\K_0|\leq C_0$. For every $x\in V(G)$, there exists some $K\in\K_0$ lying in $B_x$, equivalently $x\in A_K$. Therefore $
		V(G)=\bigcup_{K\in\K_0}A_K$, as required.
	\end{proof}
	
	\begin{theorem}\label{thm:upper-forest}
		Let $H$ be an $r$-chromatic graph, where $r\geq 3$, and suppose that $\M(H)$ contains a forest. For every $c>0$ and $d\in\mathbb N$, there is a constant $C=C(H,c,d)$ such that every $n$-vertex $H$-free graph $G$ with $\VC(G)\leq d$ and
		$\delta(G)\geq\left(\frac{r-3}{r-2}+c\right)n$
		satisfies $\chi(G)\leq C$.
	\end{theorem}
	
	\begin{proof}
		By Proposition~\ref{prop:structural}, choose a forest $F$ and $t\geq1$ such that~\eqref{eq:ambient-join} holds. If $r=3$, the result is Theorem~\ref{thm:r3}.
		
		Assume $r\geq4$. Let $C_0$ and $n_0$ be given by Lemma~\ref{lem:cover}. For $n<n_0$, the trivial bound $\chi(G)\leq n_0$ suffices. For $n\geq n_0$, choose copies $K_1,\ldots,K_m$ of $K_{r-2}(t)$, with $m\leq C_0$, such that
		\[
		V(G)=\bigcup_{i=1}^{m}A_{K_i}.
		\]
		For every $i$, the graph $G[A_{K_i}]$ is $F$-free. Otherwise, a copy of $F$ in $A_{K_i}$ together with $K_i$ would give a copy of $F\vee K_{r-2}(t)$, and hence a copy of $H$. Lemma~\ref{lem:forest-color} therefore gives $\chi(G[A_{K_i}])\leq\vtx(F)$.
		Coloring the members of the cover with disjoint palettes yields
		\[
		\chi(G)
		\leq\sum_{i=1}^{m}\chi(G[A_{K_i}])
		\leq C_0\vtx(F).\qedhere
		\]
	\end{proof}
	
	Theorem~\ref{thm:upper-forest} proves
	\[
	\dvc(H)\leq\frac{r-3}{r-2}
	\]
	whenever $\M(H)$ contains a forest. If $\M(H)$ contains no forest, then the classification theorem of Allen et al.~\cite{AllenEtAl} gives
	\[
	\delta_\chi(H)=\frac{r-2}{r-1}.
	\]
	Since $\dvc(H)\leq\delta_\chi(H)$, this also proves the required upper bound in the second case of Theorem~\ref{thm:main}.
	
	\section{The lower bound}\label{sec:lower}
	
	We now show that the standard lower-bound constructions for ordinary chromatic thresholds have uniformly bounded VC-dimension.
	
	Erd\H{o}s~\cite{ErdosProbability} proved that for every pair of positive integers $g$ and $C$, there exists a graph $G'$ with girth greater than $g$ and chromatic number greater than $C$. Fix such a graph $G'$ and write $n_0:=|V(G')|$. For an integer $\geq1$, define $G(C,g,q)$ to be the join of $G'$ and $q-1$ independent sets, for which each one has order $n_0$. Thus $G(C,g,q)$ has $q$ equally sized parts, one of which induces $G'$ and the others are independent.
	
	The following elementary observation is the standard lower-bound construction used in the ordinary chromatic-threshold problem; see Allen, B\"ottcher, Griffiths, Kohayakawa, and Morris~\cite{AllenEtAl}.
	
	\begin{proposition}\label{prop:ABGKM-construction}
		Let $H$ be a graph with $\chi(H)=r\geq3$, and let $g>|V(H)|$.
		\begin{enumerate}
			\item[(i)] For every $C$, the graph $G(C,g,r-2)$ is $H$-free.
			\item[(ii)] If $\M(H)$ contains no forest, then for every $C$, the graph $G(C,g,r-1)$ is $H$-free.
		\end{enumerate}
	\end{proposition}
	
	\begin{proof}
		Suppose on the contrary that $G(C,g,q)$ has a copy of $H$, where $q\in\{r-2,r-1\}$. Then the copy intersects $G(C,g,q)$ with the distinguished part inducing $G'$. This intersection is a forest; otherwise, it contains a cycle of length at most $|V(H)|<g$, contradicting the girth of $G'$.
		
		If $q=r-2$, then the vertices outside $G'$ lie in $r-3$ independent parts. Coloring the forest with two colors and each of these parts with one additional color would give a proper $(r-1)$-coloring of $H$, a contradiction. This proves~(i).
		
		If $q=r-1$, then the vertices outside $G'$ lie in $r-2$ independent parts. A copy of $H$ would therefore give a proper $r$-coloring for which deleting those $r-2$ color classes leaves a forest. Hence $\M(H)$ would contain a forest, proving~(ii).
	\end{proof}
	
	\begin{lemma}\label{lem:construction-vc}
		For every $C$, every $g\geq4$, and every $q\geq1$, the graph $G(C,g,q)$ has VC-dimension at most three.
	\end{lemma}
	
	\begin{proof}
		First observe that $\VC(G')\leq2$. Suppose on the contrary that three vertices $x_1,x_2,x_3$ are shattered by neighborhoods in $G'$. Let $y_{123}$ realize the trace $\{x_1,x_2,x_3\}$ and let $y_{12}$ realize the trace $\{x_1,x_2\}$. These two vertices are distinct, and $\{x_1,y_{123},x_2,y_{12}\}$
		forms a $4$-cycle of $G'$, contradicting the assumption that the girth of $G'$ is greater than $g\geq4$.
		
		Let $G:=G(C,g,q)$, and suppose that $X\subseteq V(G)$ is shattered by $\N(G)$. Since the empty trace on $X$ is realized, some vertex has no neighbor in $X$. By the join construction, this is possible only if all vertices of $X$ lie in a single part.
		
		If $X$ lies in one of the independent parts, then a vertex in this part realizes the empty trace and a vertex outside the part realizes the full trace; no other trace is possible. Hence $|X|\leq1$. Assume that $X$ lies in the part inducing $G'$. If $|X|\geq4$, then every proper subset of $X$ must be realized by a vertex of $G'$, because every vertex outside this part is adjacent to all of $X$. Fix any three-element subset $Y\subseteq X$. For every $Z\subseteq Y$, the trace $Z$ is a proper subset of $X$, and therefore is realized by a vertex of $G'$. Thus $Y$ is shattered in $G'$, contradicting $\VC(G')\leq2$. Consequently $|X|\leq3$.
	\end{proof}
	
	We can now finish the proof of the main theorem.
	
	\begin{proof}[Proof of Theorem~\ref{thm:main}]
		The upper bounds were established at the end of Section~\ref{sec:upper}.
		
		For the first lower bound, let $g>|V(H)|$ and consider $G(C,g,r-2)$. By Proposition~\ref{prop:ABGKM-construction}, this graph is $H$-free for every $C$, and its chromatic number tends to infinity with $C$. Lemma~\ref{lem:construction-vc} gives $\VC(G(C,g,r-2))\leq3$. If $r\geq4$, every vertex has at least all vertices outside its own part as neighbors, so
		\[
		\delta(G(C,g,r-2))
		\geq \frac{r-3}{r-2}|V(G(C,g,r-2))|.
		\]
		For $r=3$, the desired lower bound is zero. Hence, for every $r\geq3$,
		\[
		\dvc(H)\geq\frac{r-3}{r-2}.
		\]
		Together with Theorem~\ref{thm:upper-forest}, this proves the first case of~\eqref{eq:classification}.
		
		Now suppose that $\M(H)$ contains no forest. Proposition~\ref{prop:ABGKM-construction} shows that $G(C,g,r-1)$ is $H$-free for every $C$. It has unbounded chromatic number, VC-dimension at most three, and
		\[
		\delta(G(C,g,r-1))
		\geq\frac{r-2}{r-1}|V(G(C,g,r-1))|.
		\]
		Therefore,
		\[
		\dvc(H)\geq\frac{r-2}{r-1}.
		\]
		The reverse inequality was proved in Section~\ref{sec:upper}, completing the proof.
	\end{proof}
	
	\medskip
	\noindent\textbf{Acknowledgements.} 
	Upon finalizing this manuscript, we became aware that Lior Gishboliner, Xinqi Huang, and Hong Liu have independently obtained results identical to the main conclusions of our work. We are grateful for their notification.

\end{document}